\documentclass[11pt,reqno]{amsart}
\usepackage{amscd,amsmath,amssymb,amsthm,amsfonts,epsfig,graphics}
\usepackage{geometry}
\usepackage{graphicx}
\usepackage{mathrsfs,amssymb}
\usepackage{bm}
\usepackage{listings}

\theoremstyle{plain}

\newtheorem{corollary}{Corollary}

\newtheorem{lemma}{Lemma}

\newtheorem{problem}{Problem}
\newtheorem{proposition}{Proposition}
\newtheorem{remark}{Remark}

\newtheorem{theorem}{Theorem}
\numberwithin{equation}{section}

\newcommand{\sgn}{{\rm sgn}}

\newcommand{\ls}{\lesssim}

\begin{document}

\title[]{Sharp endpoint   estimates for eigenfunctions restricted to submanifolds of codimension 2}
\author{Xing Wang and Cheng Zhang}

\address{Department of Mathematics\\
Wayne State University\\
Detroit, MI 48202, USA}
\email{xing.wang@wayne.edu}
\address{Department of Mathematics\\
University of Rochester\\
Rochester, NY 14627, USA}
\email{czhang77@ur.rochester.edu}
\date{}
\begin{abstract}Burq-G\'erard-Tzvetkov \cite{burq} and Hu \cite{hu} established $L^p$ estimates ($2\le p\le \infty$) for the restriction of eigenfunctions to submanifolds. The estimates are sharp, except for the log loss at the endpoint $L^2$ estimates for submanifolds of codimension 2. It has long been believed that the log loss at the endpoint can be removed in general, while the problem is still open. So this paper is devoted to the study of sharp endpoint restriction estimates for eigenfunctions in this case. Chen and Sogge \cite{chensogge} removed the log loss for the geodesics on 3-dimensional manifolds. In this paper, we generalize their result to higher dimensions and prove that the log loss can be removed for totally geodesic submanifolds of codimension 2. Moreover, on 3-dimensional manifolds, we can remove the log loss for curves with nonvanishing geodesic curvatures, and more general finite type curves. The problem in 3D is essentially related to Hilbert transforms along curves in the plane and a class of singular oscillatory integrals studied by Phong-Stein \cite{ps86}, Ricci-Stein \cite{rs87}, Pan \cite{pan91}, Seeger \cite{seeger94}, Carbery-P\'erez \cite{cp99}.
\end{abstract}

\maketitle

\section{Introduction}
Let $(M,g)$ be a compact smooth $n$-dimensional Riemannian manifold and let $\Delta_g$ be the associated Laplace-Beltrami operator. Let $e_\lambda$ denote the $L^2$-normalized eigenfunction
\[-\Delta_g e_\lambda=\lambda^2e_\lambda,\]
so that $\lambda\ge0$ is the eigenvalue of the operator $\sqrt{-\Delta_g}.$

One of the main topics regarding eigenfunctions is to measure their concentration. There are several common ways to do this. The first way is by describing semi-classical (Weigner) measures, see the works by Shnirelman \cite{sh1974}, Zelditch \cite{z1987}, Colin de Verdi\`ere \cite{de1985}, G\'erard-Leichtnam \cite{ge1993}, Zelditch-Zworski \cite{zz1996}, Helffer-Martinez-Robert \cite{he1987}, Sarnak \cite{sa1995}, Lindenstrauss \cite{lin2006} and Anantharaman \cite{ana2004}. The second way is by considering the growth of the $L^p$ norms of eigenfunctions, see the works by Sogge \cite{slp, fio}, Sogge-Zelditch \cite{sz2002}, Burq-G\'erard-Tzvetkov \cite{bgt2005, bgt2004, bgt20052}, Hassell-Tacy \cite{hstc}, Hazari-Rivi\`ere\cite{heri}, Blair-Sogge \cite{bs2019}. The third way is by measuring its growth of the $L^p$ norm over some local domains, specifically, geodesic balls or tubes along geodesics, see the works by Sogge \cite{skn, sogge2016}, Blair-Sogge \cite{bs1, bs2, bs3, bs4}, Han \cite{han2015}, Hezari and Rivi\'ere \cite{he2016}. The fourth way is by considering the growth of $L^p$ norms and period integrals of eigenfunctions restricted to submanifolds, see the works by Burq-G\'erard-Tzvetkov \cite{burq}, Hu \cite{hu}, Chen \cite{chen}, Chen-Sogge \cite{chensogge}, Xi-Zhang \cite{xz}, Hezari \cite{hez}, Blair \cite{mb}, Zhang \cite{zhang}, Huang-Zhang \cite{hz2019}, Reznikov \cite{rez2015}, Sogge-Xi-Zhang \cite{sxz2017}, Canzani-Galkowski-Toth \cite{cgt2018}, Wyman \cite{wym2019}. In this paper, we study the concentration of eigenfunctions in the fourth way.

We first review the previous results. Burq-G\'erard-Tzvetkov \cite{burq} obtained the following $L^p$ estimates for eigenfunctions restricted to submanifolds. See also the works by Greenleaf-Seeger \cite{gs1994} , Tataru \cite{tata1998}, Reznikov \cite{rez} for earlier related results.

\begin{theorem}
Let ($M,g$) be a compact smooth Riemannian manifold of dimension n, and let $\Sigma$ be a smooth submanifold of dimension $k$. There exists a constant $C>0$ such that for any
$e_\lambda$, we have

           \begin{center}
           $\|e_\lambda\|_{L^p(\Sigma)}\leq C(1+\lambda)^{\rho(k,n)}\|e_\lambda\|_{L^2(M)}$
           \end{center}
           where

           \begin{equation}
          \rho(n-1,n)= \left\{ \begin{array}{l}
\frac{n-1}{2}-\frac{n-1}{p} \quad \text{if   $p_0=\frac{2n}{n-1}<p\leq+\infty$} \\
\frac{n-1}{4}-\frac{n-2}{2p} \quad \text{if   $2\leq p<p_0=\frac{2n}{n-1}$}  \\
\end{array}
\right.
          \end{equation}

          \begin{equation}
           \rho(n-2,n)=\frac{n-1}{2}-\frac{n-2}{p} \quad \text{if   $2<p\leq+\infty$}
           \end{equation}
           \begin{equation}
           \rho(k,n)=\frac{n-1}{2}-\frac{k}{p} \quad \text{if   $1\leq p\leq n-3$}.
        \end{equation}
           If $p=p_0=\frac{2n}{n-1}$ and $k=n-1$, we have
           \begin{equation}
           \|e_\lambda\|_{L^p(\Sigma)}\leq C(1+\lambda)^{\frac{n-1}{2n}} \sqrt{\log\lambda} \|e_\lambda\|_{L^2(M)}
           \end{equation}
           and if $p=2$ and $k=n-2$, we have
          \begin{equation}\label{cod2}
            \|e_\lambda\|_{L^p(\Sigma)}\leq C(1+\lambda)^{\frac{1}{2}} \sqrt{\log\lambda} \|e_\lambda\|_{L^2(M)}.
           \end{equation}
\end{theorem}

These estimates are sharp, except for the two cases: $(n,k,p)=(n,n-1,\frac{2n}{n-1})$ and $(n,k,p)=(n,n-2,2)$, which have log loss. Later, Hu \cite{hu} gave another proof of these estimates and removed the log loss for the case $(n,k,p)=(n,n-1,\frac{2n}{n-1})$, by applying the classical estimates of Fourier integral operators from the work of Greenleaf-Seeger \cite{gs1994}. However, how to remove the log loss in the remaining case $(n,k,p)=(n,n-2,2)$ is still an open problem. Therefore, this paper is devoted to the study of the sharp eigenfunction estimates in this case. It has long been believed that the log loss can be removed in general, while only one partial result is known. Recall that Chen-Sogge  \cite{chensogge} proved that if $n=3$ and the submanifold is a geodesic, then the log loss can be removed.

We first generalize Chen-Sogge's result to higher dimensions when the submanifolds are totally geodesic.

\begin{theorem}\label{thm0}

Let ($M,g$) be a compact smooth Riemannian manifold of dimension $n\ge3$, and let $\Sigma$ be a smooth totally geodesic submanifold of dimension $n-2$. There exists a constant $C>0$ such that for any
$e_\lambda$, we have
           \begin{equation}\label{key0}
           \|e_\lambda\|_{L^2(\Sigma)}\leq C(1+\lambda)^{\frac{1}{2}}\|e_\lambda\|_{L^2(M)}.
           \end{equation}
\end{theorem}
These estimates are sharp and saturated by zonal functions on the sphere $S^n$, see \cite[Section 6.2]{burq}.

Sketch of proof: The case $n=3$ was proved in \cite{chensogge}. Here we assume $n\geq 4$. First we apply the $TT^*$ argument and reduce the problem to an operator norm bound over the submanifold $\Sigma$. Then we expand the kernel of this operator by the Hadamard parametrix. For the main term, we do a scaling and compare it to a projection operator with uniform bound over $L^2(\Sigma)$. For all other terms, due to the gains on the exponent, we can use Lemma \ref{lemma2} to control their operator norms.

Next, we will focus on the 3-dimensional case where the submanifolds are smooth curves.  We will see that the problem is directly related to the singular integral operators $T_\lambda$ are of the form
\begin{equation}\label{Tlambda}T_\lambda f(t)={\rm p.v.}\int e^{i\lambda\phi(t,s)}(t-s)^{-1}a(t,s)f(s)ds,\end{equation}
where $\phi$ is smooth, $\lambda$ is real, and $a\in C_0^\infty(\mathbb{R}^2)$.
These operators and their generalizations in higher dimensions have been studied by Phong-Stein \cite{ps86}, Ricci-Stein \cite{rs87}, Pan \cite{pan91}, Seeger \cite{seeger94}, Carbery-P\'erez \cite{cp99}. It was shown by Phone-Stein \cite[p.117]{ps86} that uniform $L^2(\mathbb{R})$ estimates of $T_\lambda$ can be applied to show $L^2(\mathbb{R}^2)$ boundedness of Hilbert transform $\mathcal{H}$ along variable curves: \begin{equation}\label{ht}
  \|\mathcal{H}\|_{L^2(\mathbb{R}^2)\to L^2(\mathbb{R}^2)}\le \sup_{\lambda\in \mathbb{R}}\|T_\lambda\|_{L^2(\mathbb{R})\to L^2(\mathbb{R})}.
\end{equation} Here $\mathcal{H}$ is defined a priori on functions in $C_0^\infty(\mathbb{R}^2)$: \[\mathcal{H}f(x)=\eta(x){\rm p.v.}\int_{-\delta}^{\delta}f(x_1-t,x_2-\phi(x_1,x_1-t))\frac{dt}{t},\]
where $\eta\in C_0^\infty(\mathbb{R}^2)$ and $\delta>0$ is suitably small. Pan \cite[Theorem 2]{pan91} proved that $T_\lambda$ is uniformly bounded on $L^2(\mathbb{R})$ if one imposes a weak finite type condition: the mixed derivative $\phi_{ts}''$ does not vanish to infinite order on ${\rm supp}\ a$ (e.g. the phase function $\phi$ is real-analytic). Later, Seeger \cite{seeger94}, Carbery-P\'erez \cite{cp99} considered certain ``flat'' cases where the finite type condition is not satisfied. In the translation invariant case $\phi(t,s)=\psi(t-s)$, Nagel-Vance-Wainger-Weinberg \cite{nvww} proved necessary and sufficient conditions in the case that $\psi$ is even(or odd) and convex. However, Nagel-Wainger \cite[Theorem 4.1]{nw} constructed an odd smooth function $\psi(t)$ on $[-1,1]$, which vanishes of infinite order at $t=0$, such that the Hilbert transform $\mathcal{H}$ along the curve $(t,\psi(t))$ is unbounded on $L^2(\mathbb{R}^2)$. This implies that the operators $T_\lambda$ may not be uniformly bounded on $L^2(\mathbb{R})$ if the finite type condition in \cite[Theorem 2]{pan91} is removed.

In the following, we extend Chen-Sogge's result \cite[Theorem 1]{chensogge} on geodesics to more general curves.
\begin{theorem}\label{thm1}
Let $(M,g)$ be a compact smooth 3-dimensional Riemannian manifold. Let $\gamma\subset M$ be a fixed unit-length curve with nonvanishing geodesic curvatures. There exists a constant $C>0$ such that for any
$e_\lambda$, we have
\begin{equation}\label{key}\|e_\lambda\|_{L^2(\gamma)}\le C(1+\lambda)^\frac12\|e_\lambda\|_{L^2(M)}.\end{equation}
\end{theorem}

 This bound is sharp and saturated by zonal functions on the sphere $S^3$, see \cite[Section 6.2]{burq}. With more careful consideration as in the work of Pan \cite{pan91}, the log loss can be removed for those curves satisfying certain finite type condition, such as the real-analytic curves on real-analytic manifolds.  We use the wave kernel method and the Hadamard parametrix to reduce the problem to the uniform $L^2$-estimates of a class of singular integrals with oscillatory terms in \eqref{Tlambda}. Then we may use the H\"ormander's oscillatory integral theorem to estimate the operator norm. For more general curves, one may need to apply the oscillatory integral theorem in Pan \cite[Theorem 2]{pan91}.

  The paper is organized as follows. In section 2, we first recall the Hadamard parametrix and oscillatory integral estimates from \cite{burq}. Then we use them to prove Theorem \ref{thm0}. In section 3, we prove Theorem \ref{thm1} by the analyzing precisely the oscillations which appear in the phase (i.e. the distance function restricted to the curve). In section 4, we discuss the possible generalizations and the related open problems. Throughout this paper, the injectivity radius of $M$ is sufficiently large, and using a partition of unity, we may assume that $\Sigma$ is  contained in the domain of a given coordinate patch. The notation $A\lesssim B$ and $ A\gtrsim B$
denote $A\leq CB$ and $A\geq CB$ respectively,  for some constant $C$ which does not depend on $\lambda$.

\noindent\textbf{Acknowledgement. } The authors would like to thank Professor Allan Greenleaf, Professor Chris Sogge, and Professor Yibiao Pan for their helpful suggestions and comments. Thanks also to Xiaoqi Huang for his very thorough reading of the preprint.

\section{Restriction to totally geodesic submanifolds}
   \label{pr}
In this section, we prove Theorem \ref{thm0}. First, we need the Hadamard parametrix, see \cite[Theorem 3.1.5]{hangzhou} for references.

 \begin{lemma}
  \label{lemma1}
  Let $(M,g)$ be a compact manifold without boundary.
  Let  $\delta>0$ be smaller than the injectivity radius of $(M,g)$.
  If $|t |\leq \delta$ and $N>n+3$, then we have
   \begin{equation}
  \cos t\sqrt{-\Delta_g}(x,y)=K_N(t,x,y)+R_N(t,x,y)
 \end{equation}
  where $R_N \in C^{N-n-3}([-\delta,\delta]\times M\times M)$, and

 \begin{equation}
          K_N(t,x,y)= \left\{ \begin{array}{l}
\partial_t(\sum\limits_{\nu=0}^{N} \omega_\nu (x,y)E_\nu (t,\kappa(x,y))) \quad\quad\quad\  \text{if $t \geq 0$} \\
-\partial_t(\sum\limits_{\nu=0}^{N} \omega_\nu (x,y)E_\nu (-t,\kappa(x,y))) \quad\quad \text{if $t<0$}.
\end{array}
\right.
          \end{equation}
Here $\kappa(x,y)$ is the vector from $x$ to $y$ in the local geodesic coordinates at $x$, and we have $|\kappa(x,y)|=d_g(x,y)$. Moreover, $\omega_\nu \in C^\infty(M\times M)$, and $\omega_0(x,x)=1,\forall x \in M$.
$E_\nu$ are distributions such that
\begin{equation}
 \partial_tE_0(t,x)=\frac{H(t)}{(2\pi)^{n}}\int_{\mathbb{R}^n} e^{ix\cdot \xi}\cos(t|\xi|)d\xi,
 \end{equation}
  and $E_\nu, \nu=1,2,3...$ is a finite linear combination of Fourier integrals of the form:
 \begin{equation}
H(t)t^j\int_{|\xi|>1} e^{ix\cdot \xi\pm it|\xi|}|\xi|^{-\nu-1-k}d\xi+\eta_{j\nu}
 \end{equation}
where $j,\ k\ge0$, $j+k=\nu$, and $\eta_{j\nu}$ are smooth. Here $H(t)$ is the Heaviside function. Furthermore, we have $\partial_t E_\nu=\frac{t}{2}E_{\nu-1}$, $\nu=1,2,3,...$.
 \end{lemma}
Hence, for $0<t<\delta$ and $N>n+3$, modulo a smooth error we can write
\begin{align*}
  \cos t\sqrt{-\Delta_g}(x,y)&=\omega_0(x,y)\partial_tE_0(t,\kappa(x,y))+\sum_{\nu=1}^N\omega_\nu(x,y)\frac{t}2 E_{\nu-1}(t,\kappa(x,y))\\
  &=\frac{\omega_0(x,y)}{(2\pi)^n}\int_{\mathbb{R}^n} e^{i\kappa(x,y)\cdot \xi}\cos t|\xi|d\xi\\
  &+\sum_{\nu=1}^N\sum_{j=0}^{\nu-1}\frac{\omega_\nu(x,y)}{(2\pi)^n}a_{j\nu}^\pm t^{j+1}\int_{|\xi|>1}e^{i\kappa(x,y)\cdot \xi\pm it|\xi|}|\xi|^{-2\nu+1+j}d\xi.
\end{align*}
 Here $a_{j\nu}^\pm$ are constant coefficients. We also need the following estimates for oscillatory integral operators, which will be used several times in the proof. See \cite[Proposition 6.3]{burq}.

  \begin{lemma}
  \label{lemma2}Let $m\ge0$ and $k\ge1$.
  Let $(N,h)$ be a $k$ dimensional compact Riemannian manifold with distance function $d_h:N\times N\to [0,\infty)$. Let $Q_\lambda$ be an operator on $L^2(N)$ with the kernel $Q_\lambda(x,y)$ satisfying $|Q_\lambda(x,y)|\ls 1$ and when $d_h(x,y)>\lambda^{-1}$
  \begin{equation}
  Q_\lambda(x,y)= \sum\limits_{\pm} \frac{e^{\pm i\lambda d_h(x,y)}}{(\lambda d_h(x,y))^m}a^\pm(x,y,\lambda),
     \end{equation}
     where $a^\pm(x,y,\lambda) \in C^\infty(N\times N\times \mathbb{R})$ and we have $|\partial^\alpha_{x,y}a^\pm(x,y,\lambda)|\le C_\alpha d_h(x,y)^{-|\alpha|}$, for $d_h(x,y)>\lambda^{-1}$. Then
     \begin{eqnarray}
     \|Q_\lambda\|_{L^2(N)\to L^2(N)}
     &&\lesssim \lambda^{-m-\frac{k-1}{2}} \sum\limits_{j\leq \log\lambda}^{}2^{j(m-\frac{k+1}{2})} \nonumber \\
        &&\lesssim \left\{ \begin{array}{l}
\lambda^{-k} \quad\quad\quad\quad\quad\quad\quad\ \text{if $m>\frac{k+1}{2}$} \\
\lambda^{-m-\frac{k-1}{2}} \quad\quad\quad\quad\quad \text{if $m<\frac{k+1}{2}$} \\
\lambda^{-m-\frac{k-1}{2}}\log\lambda \quad\quad\ \ \  \text{if $m=\frac{k+1}{2}$}.
\end{array}
\right.\nonumber
          \end{eqnarray}
  \end{lemma}

\subsection{Proof of Theorem \ref{thm0}}

 Choose any
   $\chi \in \mathcal{S}(\mathbb{R})$, such that $\chi(0)=1$, supp$\hat{\chi} \subset [1,2]$.
Let $\chi_\lambda f=\chi(\lambda-\sqrt{-\Delta_g})f$, then $\chi_\lambda \varphi_\lambda=\varphi_\lambda$. Thus it suffices to show
\begin{equation}
\label{eq1}
||\chi_\lambda||_{L^2(M)\to L^2(\Sigma)} \lesssim \lambda^{\frac{1}{2}}.
\end{equation}
By $TT^*$ argument,  \eqref{eq1} is equivalent to
\begin{equation}\label{bd}
||\chi_\lambda \chi_\lambda^*||_{L^2(\Sigma)\to L^2(\Sigma)} \lesssim \lambda.
\end{equation}
Let $T_\lambda=\chi_\lambda \chi_\lambda^*$, a simple calculation shows the kernel of $T_\lambda$ is the same as
\begin{equation}
\chi^2(\lambda-\sqrt{-\Delta_g})(x,y)|_{\Sigma\times\Sigma}
\end{equation}

Let $\phi=\chi^2$, then $\phi(0)=1$, supp$\hat{\phi}\subset [2,4]$.
\begin{eqnarray}
T_\lambda&=&\phi(\lambda-\sqrt{-\Delta_g}) \nonumber\\
&=&\frac{1}{2\pi}\int_{\mathbb{R}}\hat{\phi}(t)e^{it(\lambda-\sqrt{-\Delta_g})}dt \nonumber\\
&=&\frac{1}{\pi}\int_{\mathbb{R}}\hat{\phi}(t)e^{i\lambda t}\cos(t\sqrt{-\Delta_g})dt-\phi(\lambda+\sqrt{-\Delta_g}).\nonumber
\end{eqnarray}
Here $\phi(\lambda+\sqrt{-\Delta_g})(x,y)=O(\lambda^{-N})$, $N=1,2,...$. Then by the support property of $\hat\phi$ and the Hadamard parametrix in Lemma \ref{lemma1}, modulo $O(1)$ the kernel of $T_\lambda$ can be written as
\begin{eqnarray}
T_\lambda(x,y)&=&\frac{\omega_0(x,y)}{\pi(2\pi)^{n}}\int_\mathbb{R}\hat{\phi}(t)e^{i\lambda t}\int_{\mathbb{R}^n}e^{i\kappa(x,y)\cdot \xi}\cos t|\xi| d\xi dt \nonumber\\
&&+\sum\limits_{\nu=1}^{N} \sum\limits_{j=0}^{\nu-1}a_{j\nu}^\pm\frac{\omega_\nu(x,y)}{(2\pi)^{n}}\int_\mathbb{R}\hat{\phi}(t)e^{i\lambda t}  t^{j+1}\int_{|\xi|>1} e^{i\kappa(x,y)\cdot \xi\pm it|\xi|}|\xi|^{-2\nu+1+j}d\xi dt \nonumber\\
&=&\frac{\omega_0(x,y)}{(2\pi)^n}\int_{\mathbb{R}^n}\phi(\lambda\pm|\xi|)e^{i\kappa(x,y)\cdot \xi}d\xi \nonumber\\
&&+\sum\limits_{\nu=1}^{N}\sum\limits_{j=0}^{\nu-1}\frac{\omega_\nu(x,y)}{(2\pi)^{n-1}}\int_{|\xi|>1}\phi_{j\nu}^\pm(\lambda\pm|\xi|)e^{i\kappa(x,y)\cdot \xi}|\xi|^{-2\nu+1+j}d\xi.
\label{original}
\end{eqnarray}
Here $\phi_{j\nu}^\pm$ is the inverse Fourier transform of $a_{j\nu}^\pm \hat{\phi}(t)t^{j+1} $, which are also Schwartz functions independent of $\lambda$.

Next we introduce a new operator which will play an important role in the proof and help us simplify the calculations. Define $S_{r}^{\nu}$, $\nu=0,1,2,3...$ to be the operator with kernel:
\begin{equation}
S_{r}^{\nu}(x,y)=\omega_\nu(x,y)\int_{S^{n-1}}e^{ir\kappa(x,y)\cdot\omega}d\omega.
\end{equation}
Here $S^{n-1}$ denotes the standard unit sphere in $\mathbb{R}^n$. By Stationary Phase (e.g. \cite{fio}), we can see that $S_{r}^{\nu}$ satisfies the condition in Lemma \ref{lemma2} with $k=n-2$ and $m=\frac{n-1}{2}$, thus by Lemma \ref{lemma2}, we have the following estimate:
\begin{equation}
\label{sbound}
\|S_{r}^{\nu}\|_{L^2(\Sigma)\to L^2(\Sigma)} \lesssim r^{2-n}\log{r}.
\end{equation}
Then using the spherical coordinates for the $\xi$ variables, we can rewrite \eqref{original} as
\begin{align*}
T_\lambda(x,y)&=(2\pi)^{-n}\int_{0}^{\infty}\phi(\lambda\pm r)S_r^0(x,y)r^{n-1}dr \nonumber\\
&+(2\pi)^{1-n}\sum\limits_{\nu=1}^{N}\sum\limits_{j=0}^{\nu-1}\int_{1}^{\infty}\phi_{j\nu}^\pm(\lambda\pm r)S_r^\nu(x,y)r^{-2\nu+1+j}\cdot r^{n-1}dr \nonumber\\
&:=A_\lambda(x,y)+B_\lambda(x,y).
\end{align*}
 By using the estimate \eqref{sbound}, we are able to control $B_\lambda$:
\begin{align*}
\|B_\lambda\|_{L^2(\Sigma)\to L^2(\Sigma)} &\lesssim \sum\limits_{\nu=1}^{N}\sum\limits_{j=0}^{\nu-1}\int_{1}^{\infty}|\phi_{j\nu}^\pm(\lambda\pm r)|\|S_r^\nu\|_{L^2(\Sigma)\to L^2(\Sigma)}r^{-1}\cdot r^{n-1}dr\\
&\lesssim \sum\limits_{\nu=1}^{N}\sum\limits_{j=0}^{\nu-1}\int_{1}^{\infty}|\phi_{j\nu}^\pm(\lambda\pm r)|r^{2-n}\log r\cdot r^{n-2}dr \\
&\lesssim \log\lambda.
\end{align*}
It is better than \eqref{bd}.  Therefore, if we can improve \eqref{sbound} for $S_r^0$ and obtain the following stronger estimate without the log loss, then we are able to control $A_\lambda$ as needed, and the proof of \eqref{bd} would be complete.
\begin{lemma}\label{keylemma}
We have
\begin{equation}
\|S_r^0\|_{L^2(\Sigma)\to L^2(\Sigma)} \lesssim r^{-n+2}.
\label{kbound}
\end{equation}
\end{lemma}
We will see that it is true by exploiting the fact that $\omega_0\equiv1$ on the diagonal and the assumption that $\Sigma$ is totally geodesic.

\subsection{Proof of Lemma \ref{keylemma}}
Since this estimate is a local estimate, without loss of generality, we may assume that $\Sigma$ is closed. Let $h=g|_\Sigma$, then $(\Sigma, h)$ is a closed Riemannian manifold.

Denote $\tilde{\kappa}(x,\cdot):\ \Sigma\to
\mathbb{R}^{n-2}$, $y\mapsto \tilde{\kappa}(x,y)$ as the vector from $x$ to $y$ in the local
geodesic coordinates with respect to $(\Sigma,h)$ at $x$. Similarly, in the following, any functions or operators under '$\sim$' will be on the submanifold $\Sigma$.
Since  $\Sigma$ is totally geodesic, we can assume
$\kappa|_{\Sigma\times \Sigma}=(\tilde\kappa,0,0)$.
Accordingly, we can make the following change of coordinates:
\begin{equation}
B^{n-2}(1)\times [0,2\pi)\to S^{n-1}:
 (z,\sqrt{1-|z|^2}\cos\theta,\sqrt{1-|z|^2}\sin\theta)
\end{equation}
The Jacobian is 1, thus we can modify the kernel of operator $S_r^0$ as
\begin{eqnarray}
S_r^0(x,y)&=&\omega_0(x,y)\int_{B^{n-2}(1)}\int_{0}^{2\pi}e^{ir\kappa(x,y)\cdot
\omega(z,\theta)}d\theta dz \nonumber\\
&=&2\pi\omega_0(x,y)\int_{B^{n-2}(1)}e^{ir\tilde\kappa(x,y)\cdot z}dz \nonumber\\
&=&2\pi\omega_0(x,y)r^{2-n}\int_{B^{n-2}(r)}e^{i\tilde\kappa(x,y)\cdot z}dz. \label{sph}
\end{eqnarray}
Here $B^{n-2}(r)$ is the ball of radius $r$ centered at 0 in $\mathbb{R}^{n-2}$.

Let $\overline{S}_r$ be the operator with kernel
\begin{equation}
\overline{S}_r(x,y)=\omega_0(x,y)\int_{B^{n-2}(r)}e^{i\tilde\kappa(x,y)\cdot z}dz.
\end{equation}
Then \eqref{kbound} is equivalent to
\begin{equation}
\|\overline{S}_r\|_{L^2(\Sigma)\to L^2(\Sigma)} \lesssim 1.
\end{equation}

To prove this estimate, we compare $\overline{S}_r$ to an operator with uniform
bound over $L^2(\Sigma)$.

Consider the eigenfunctions and eigenvalues of $-\Delta_h$ over $\Sigma$:
\begin{equation}
-\Delta_h e_{\mu_j}=\mu_j^2 e_{\mu_j},\ 0=\mu_0<\mu_1\leq \mu_2\leq\cdots.
\end{equation}
For $\mu\ge1$, let $P_\mu$ be the projection map to the eigenspace with
eigenvalue $\leq \mu$, that is
\begin{equation}
P_\mu=\sum\limits_{\mu_j\leq\mu}^{}E_{\mu_j}=
\textbf{1}_{[-\mu,\mu]}(\sqrt{-\Delta_h}).
\end{equation}
Obviously, $\|P_\mu\|_{L^2(\Sigma)\to L^2(\Sigma)} \leq 1$. Note that we may rewrite $P_\mu$ by the Fourier inversion formula
\begin{eqnarray}
P_\mu&=&\frac{1}{2\pi}\int_{\mathbb{R}}e^{it\sqrt{-\Delta_h}}
\hat{\textbf{1}}_{[-\mu,\mu]}(t)dt \nonumber\\
&=&\frac{1}{\pi}\int_{\mathbb{R}}e^{it\sqrt{-\Delta_h}}\frac{\sin \mu t}{t}dt \nonumber\\
&=&\frac{1}{\pi}\int_{\mathbb{R}}e^{it\sqrt{-\Delta_h}}\beta(t)\frac{\sin \mu t}{t}dt+\frac{1}{\pi}\int_{\mathbb{R}}e^{it\sqrt{-\Delta_h}}(1-\beta(t))\frac{\sin \mu t}{t}dt.
\label{haha}
\end{eqnarray}
Here we choose $\beta(t)$ to be an even cut-off function supported in $[-\delta,\delta]$, $\beta(t)=1$ in $[-\frac{\delta}{2},\frac{\delta}{2}]$, $\delta>0$ is less than the injective radius of $(\Sigma, h)$.
Let $r_\mu$ be the inverse Fourier transform of $t\mapsto (1-\beta(t))\frac{2\sin \mu t}{t}$, by integration by parts
we see that $r_\mu$ satisfies
\begin{equation}\label{rmu}
|r_\mu(\tau)|\leq C_N(1+|\mu-|\tau||)^{-N}, \ \mu \geq 1, \ N=1,2,3\cdots.
\end{equation}
Hence we can rewrite \eqref{haha} as
\begin{equation}\label{haha2}
P_\mu=\frac{1}{\pi}\int_{\mathbb{R}}\beta(t)\frac{\sin \mu t}{t}\cos t\sqrt{-\Delta_h}dt+
r_\mu(\sqrt{-\Delta_h}).
\end{equation}
For the second term in \eqref{haha2}, by \eqref{rmu} it is a multiplier uniformly bounded over $L^2(\Sigma)$.
 For the first term, we can compute it by the Hadamard parametrix on $(\Sigma, h)$. So modulo $O(1)$ we can rewrite the first term in \eqref{haha2} as
\begin{eqnarray}
&&\frac{\tilde{\omega}_0(x,y)}{\pi(2\pi)^{n-2}}\int_{\mathbb{R}}
\int_{\mathbb{R}^{n-2}}\beta(t)\frac{\sin \mu t}{t}e^{i\tilde\kappa(x,y)\cdot z}\cos t|z|dtdz
\nonumber\\
&&+\sum_{\nu=1}^{N}\sum_{j=0}^{\nu-1}\frac{\tilde\omega_\nu(x,y)}{(2\pi)^{n-2}}a_{j\nu}^\pm \int_\mathbb{R}\beta(t)\frac{\sin \mu t}{t} t^{j+1}\int_{|z|>1} e^{i\kappa(x,y)\cdot z\pm it|z|}|z|^{-2\nu+1+j}dz dt  \nonumber\\
&=&\frac{\tilde{\omega}_0(x,y)}{\pi(2\pi)^{n-2}}\int_{\mathbb{R}}
\int_{\mathbb{R}^{n-2}}\beta(t)\frac{\sin \mu t}{t}e^{i\tilde\kappa(x,y)\cdot z}e^{it|z|}dtdz
\nonumber\\
&&+\sum_{\nu=1}^{N}\sum_{j=0}^{\nu-1}\frac{\tilde\omega_\nu(x,y)}{(2\pi)^{n-2}}a_{j\nu}^\pm \int_\mathbb{R}\beta(t)t^{j}\sin \mu t \int_{|z|>1} e^{i\kappa(x,y)\cdot z\pm it|z|}|z|^{-2\nu+1+j}dz dt  \nonumber
\end{eqnarray}
\begin{eqnarray}&=&\frac{\tilde{\omega}_0(x,y)}{(2\pi)^{n-2}}
\int_{\mathbb{R}^{n-2}}\textbf{1}_{[-\mu,\mu]}(|z|)e^{i\tilde\kappa(x,y)\cdot z}dz
\nonumber\\
&&-\frac{\tilde{\omega}_0(x,y)}{(2\pi)^{n-2}}
\int_{\mathbb{R}^{n-2}}r_\mu(|z|)e^{i\tilde\kappa(x,y)\cdot z}dz \nonumber\\
&&+\sum_{\nu=1}^{N}\sum_{j=0}^{\nu-1}\frac{\tilde\omega_\nu(x,y)}{(2\pi)^{n-2}}a_{j\nu}^\pm \int_\mathbb{R}\beta(t)t^{j}\sin \mu t \int_{|z|>1} e^{i\kappa(x,y)\cdot z\pm it|z|}|z|^{-2\nu+1+j}dz dt  \nonumber\\
&:=&\tilde{A}_\mu-\tilde{B}_\mu+\tilde{C}_\mu.\nonumber
\end{eqnarray}
For $\tilde{B}_\mu$, using spherical coordinates, we know that
\[
\tilde{B}_\mu=(2\pi)^{2-n}
\int_{0}^{\infty}r_\mu(\rho)\tilde S_{\rho}^{0}(x,y)\rho^{n-3}  d\rho\]
where
\begin{equation}
\tilde S_{\rho}^{\nu}(x,y)=\tilde\omega_\nu(x,y)\int_{S^{n-3}}e^{ir\tilde\kappa(x,y)\cdot \theta}d\theta.
\end{equation}
Then by Stationary Phase, we can see that $\tilde S_{r}^{\nu}$ satisfies the condition in Lemma \ref{lemma2} with $k=n-2$ and $m=\frac{n-3}{2}$. Thus by Lemma \ref{lemma2}, we have the following estimate:
\begin{equation}\label{tildes}
\|\tilde S_{\rho}^{\nu}\|_{L^2(\Sigma)\to L^2(\Sigma)} \lesssim \rho^{3-n}.
\end{equation}
So
\begin{equation}
\|\tilde B_\mu\|_{L^2(\Sigma)\to L^2(\Sigma)}\lesssim \int_{0}^{\infty}(1+|\mu-\rho|)^{-N}
 \rho^{3-n} \rho^{n-3}d\rho \lesssim 1.
\end{equation}
If we denote the Fourier transform of $a_{j\nu}^\pm\beta(t)t^j$ by $\psi_{j\nu}$, which are Schwartz functions, then $\tilde{C}_\mu$ is a finite sum of
\[\int_1^\infty \psi_{j\nu}(\mu\pm \rho)\tilde S_{\rho}^\nu(x,y)\rho^{-2\nu+1+j}\rho^{n-3}d\rho.\]
Hence by \eqref{tildes} we have
\begin{equation}
\|\tilde C_\mu\|_{L^2(\Sigma)\to L^2(\Sigma)}\ls \sum_{\nu=1}^{N}\sum_{j=0}^{\nu-1}\int_1^\infty \psi_{j\nu}(\mu\pm \rho)\rho^{3-n}\rho^{-1}\rho^{n-3}d\rho\ls \mu^{-1}\le1.
\end{equation}

From above we know that $\tilde{A}_\mu=P_\mu+\tilde B_\mu-\tilde C_\mu-r_\mu(\sqrt{-\Delta_h})$ is uniformly bounded on $L^2(\Sigma)$, so is $\overline{P}_\mu:=(2\pi)^{n-2}\tilde{A}_\mu$. Let $\mu=r$ and consider the difference between $\overline{S}_r$
and $\overline{P}_r$:
\begin{align*}
\overline{S}_r(x,y)-\overline{P}_r(x,y)&=(\omega_0(x,y)-\tilde\omega_0(x,y))\int_{B^{n-2}(r)}e^{i\tilde\kappa(x,y)\cdot z}dz\\
&=(\omega_0(x,y)-\tilde\omega_0(x,y))\frac{r^{n-2}}{2\pi}\int_{S^{n-1}}e^{i\kappa(x,y)\cdot \omega}d\omega.
\end{align*}
Here we use \eqref{sph} in the second equality.
Since $\omega_0(x,x)=\tilde\omega_0(x,x)=1$, we have
\[\omega_0(x,y)-\tilde\omega_0(x,y)=O(d_h(x,y))=O(d_g(x,y)).\]Then by Stationary Phase and Lemma \ref{lemma2} with $k=n-2$ and $m=\frac{n-3}2$, we know
\begin{equation}
 \|\overline{S}_r-\overline{P}_r\|_{L^2(\Sigma)\to L^2(\Sigma)}\lesssim 1.
\end{equation}
Finally we conclude that
\begin{equation}
 \|\overline{S}_r\|_{L^2(\Sigma)\to L^2(\Sigma)}\lesssim 1,
\end{equation}
and this completes the proof.
\section{Restriction to curves on 3-d manifolds}

In this section, we prove Theorem \ref{thm1}.  Let $\gamma:[-\frac12,\frac12]\to M$ be a smooth curve segment parametrized by arc length. Using Taylor expansion, we have the following precise description of $d_g(\gamma(t),\gamma(s))$. See \cite[Lemma 4.5]{burq}.
\begin{lemma}\label{dlemma} We can write for $t,s\in [-\frac12,\frac12]$,
\begin{equation}\label{dist}d_g(\gamma(t),\gamma(s))=|t-s|(1-c(t)(t-s)^2+d(t,t-s)(t-s)^3),\end{equation}
where $c(t)$ and $d(t,t-s)$ are smooth functions. And $c(t)\ge c_0>0$ if $\gamma$ has nonvanishing geodesic curvatures.
\end{lemma}
Note that \cite{burq} only proved it for curves on surfaces, but their proof also works for curves on $3$-dimensional manifolds.
Thus, if  $\phi(t,s)=d_g(\gamma(t),\gamma(s))\sgn(t-s)$ then $\phi(t,s)$ is smooth.
 With this lemma, we are ready to prove Theorem \ref{thm1}.

 Let $\rho\in {\mathcal{S}}(\mathbb R)$ such that $\rho(0)=1$ and ${\rm supp}\ \hat\rho\subset[-1/2,1/2]$, then it is clear that the operator $\rho(\lambda-\sqrt { - {\Delta _g}})$ reproduces eigenfunctions, namely
\[\rho(\lambda-\sqrt { - {\Delta _g}})e_\lambda=e_\lambda.\]
Let $\chi=|\rho|^2$. After a standard $TT^*$ argument, we only need to prove
\begin{equation}\label{cmpmain}
\|\chi(\lambda-\sqrt { - {\Delta _g}})\|_{L^2(\gamma)\rightarrow L^{2}(\gamma)}\ls \lambda.
\end{equation}
We rewrite the kernel $\chi(\lambda-\sqrt { - {\Delta _g}})(x,y)$ by the Fourier inversion formula
\[\begin{aligned}\chi (\lambda  - \sqrt { - {\Delta _g}})(x,y) &= \frac{1}
{{2\pi }}\int_{\mathbb{R}} \hat \chi (\tau ){e^{i\lambda \tau }}({e^{ - i\tau \sqrt { - {\Delta _g}}}})(x,y)d\tau\\
&= \frac{1}
{{\pi }}\int_{\mathbb{R}} \hat \chi (\tau){e^{i\lambda \tau }}{\cos(\tau \sqrt { - {\Delta _g}})}(x,y)d\tau+O(\lambda^{-N}).
\end{aligned}
\]
Let \[K(t,s)=\frac{1}
{{\pi }}\int_{\mathbb{R}} \hat \chi (\tau){e^{i\lambda \tau }}{\cos(\tau \sqrt { - {\Delta _g}})}(\gamma(t),\gamma(s))d\tau.\]
Since ${\rm supp}\hat{\chi}\subset[-1,1]$ and the injectivity radius of $(M,g)$ is sufficiently large, we may use the Hadamard parametrix to estimate $K(t,s)$. Indeed, as in \cite[(2.9)]{chensogge} we can write
\[\begin{aligned}\cos(\tau\sqrt{-\Delta_g})(\gamma(t),\gamma(s))&=(2\pi)^{-3}w(t,s)\int_{\mathbb{R}^3}e^{id_g(\gamma(t),\gamma(s))\xi_1}\cos(\tau|\xi|)d\xi\\
&+\sum_{\pm}\int_{\mathbb{R}^3}e^{id_g(\gamma(t),\gamma(s))\xi_1}e^{\pm i\tau|\xi|}a_\pm(\tau,t,s;|\xi|)d\xi+R(t,s),
\end{aligned}
\]
where \begin{equation}\label{wts}w\in C^\infty[-\tfrac12,\tfrac12]^2,\ w(t,t)=1,\ |t|\le \tfrac12,\end{equation}
\[|\partial_\tau^ja_\pm(\tau,t,s;|\xi|)|\le C_j(1+|\xi|)^{-2}, \ j=0,1,2,...,\]
\[|R(t,s)|\le C.\]
It is not difficult to see that the contributions of the second term and the remainder term to $K(t,s)$ is $O(1)$, since
\begin{align*}\Big|\int_{\mathbb{R}^3}\int_{\mathbb{R}} \hat\chi(\tau)e^{id_g(\gamma(t),\gamma(s))\xi_1}a_\pm(\tau,t,s;|\xi|)e^{i\lambda\tau}e^{\pm i\tau|\xi|}d\tau d\xi\Big|\\
\le C_N\int_{\mathbb{R}^3}(1+|\lambda-|\xi|)^{-N}(1+|\xi|)^{-2}d\xi\ls 1.\end{align*}
 Then we have
\[\begin{aligned}K(t,s)&=(2\pi)^{-4}w(t,s)\int_{\mathbb{R}}\int_{\mathbb{R}^3}\hat\chi(\tau)e^{id_g(\gamma(t),\gamma(s))\xi_1}e^{i\tau(\lambda-|\xi|)}d\xi d\tau+O(1)\\
&=(2\pi)^{-3}w(t,s)\int_{\mathbb{R}^3}\chi(\lambda-|\xi|)e^{id_g(\gamma(t),\gamma(s))\xi_1}d\xi+O(1)\\
&=\frac{w(t,s)}{2\pi^2}\int_0^\infty\chi(\lambda-r)\frac{\sin(d_g(\gamma(t),\gamma(s))r)}{d_g(\gamma(t),\gamma(s))}rdr+O(1)\\
&=\frac{w(t,s)}{2\pi^2}\int_0^\infty\chi(\lambda-r)\frac{\sin(d_g(\gamma(t),\gamma(s))r)}{|t-s|}rdr+O(\lambda)\ \ \ (by \ Lemma\ \ref{dlemma})\\
&=\frac1{2\pi^2}\int_0^\infty\chi(\lambda-r)\frac{\sin(d_g(\gamma(t),\gamma(s))r)}{|t-s|}rdr+O(\lambda)\ \ \ (by\ \eqref{wts}).
\end{aligned}\]
Let \[T_\lambda f(t)={\rm p.v.}\int_{-\frac12}^\frac12\frac{e^{i\lambda\phi(t,s)}}{t-s}f(s)ds,\]
where $\phi(t,s)=d_g(\gamma(t),\gamma(s)){\rm sgn}(t-s)$.

Therefore, we will have \eqref{cmpmain}
if we can prove
\begin{lemma}\label{lemma5}
  $T_\lambda$ is uniformly bounded on $L^2[-\frac12,\frac12]$.
\end{lemma}
To prove this lemma, we need to exploit the curvature assumption on the curve $\gamma$.
\subsection{Proof of Lemma \ref{lemma5}}
 Let $\beta\in C_0^\infty(\mathbb{R})$ be an even function satisfying \[{\rm supp}\beta\subset [-2,-\tfrac12]\cup[\tfrac12,2],\ \sum_{j=0}^\infty \beta(2^jx)=1,\ x\in[-1,1].\] Let $T_\lambda=\sum_{2^j\le\lambda^{1/3}}T_\lambda^j+R_\lambda$, where
\[T_\lambda^j f(t)=\int \frac{e^{i\lambda\phi(t,s)}}{t-s}\beta(2^j(t-s))f(s)ds,\ j=0,1,2,...,\]
\[R_\lambda f(t)=\sum_{2^j\ge\lambda^{1/3}}T_\lambda^jf(t)={\rm p.v.}\int \frac{e^{i\lambda\phi(t,s)}}{t-s}\eta(\lambda^\frac13(t-s))f(s)ds,\ \eta\in C_0^\infty(\mathbb{R}).\]
First, we show that $R_\lambda$ is uniformly bounded on $L^2(\mathbb{R})$ by comparing it with a convolution operator. Indeed, the boundedness of the Fourier multiplier
\[\Big|\Big(e^{i\lambda x}\eta(\lambda^\frac13x){\rm p.v.}\frac1x\Big)^\wedge(y)\Big|\lesssim\|\hat\eta\|_1\approx 1,\]\
implies the uniform $L^2\to L^2$ bound of the convolution operator:
\[f \longmapsto {\rm p.v.}\int \frac{e^{i\lambda(t-s)}}{t-s}\eta(\lambda^\frac13(t-s))f(s)ds.\]
Then the uniform bound $\|R_\lambda\|_{L^2\to L^2}\ls 1$ follows from
\[\int \frac{|e^{i\lambda\phi(t,s)}-e^{i\lambda(t-s)}|}{|t-s|}|\eta(\lambda^\frac13(t-s))|ds\lesssim \int_{|t-s|\le\lambda^{-1/3}} \lambda|t-s|^2ds\lesssim 1,\] and Young's inequality.

Next, we need to handle $T_\lambda$. Young's inequality clearly gives $\|T_\lambda^j\|_{L^2\to L^2}\ls 1$, which is not precise enough to imply $\|T_\lambda\|_{L^2\to L^2}\ls 1$. We must refine it for $ 2^j\le\lambda^\frac13$ by using the curvature condition. Recall that \[\phi(t,s)=(t-s)-c(t)(t-s)^3+d(t,t-s)(t-s)^4.\] Direct calculation gives
 \[\phi_{ts}''(t,s)=6c(t)(t-s)+O(|t-s|^2),\]
 \[\phi_{tts}'''(t,s)=6c(t)+O(|t-s|),\]
 \[\phi_{ttts}''''(t,s)=O(1).\]
 Let $\beta^+(2^j(t-s))=\beta(2^j(t-s))\textbf{1}_{t>s}$ and $\beta^-(2^j(t-s))=\beta(2^j(t-s))\textbf{1}_{t<s}$, which are smooth functions and their sum is equal to $\beta(2^j(t-s))$. In the following, we only deal with the operators with $\beta^+$, and the same argument also works for those  with  $\beta^-$.
 Let \[ \tilde T_\lambda^{j}f(t)=\int e^{i\lambda\phi(t,s)}a(t,s)f(s)ds,\]
where $a(t,s)=(t-s)^{-1}\beta^+(2^j(t-s))$.  The kernel of $\tilde T_\lambda^j \tilde T_\lambda^{j*}$ is
\[\begin{aligned}K(s,s')&=\int {{e^{i\lambda (\phi (t,s) - \phi (t,s'))}}a(t,s)\overline {a(t,s')} dt}\\
&:=\int {{e^{i\lambda (s-s')\varphi(t,s,s')}}\tilde a(t,s,s')} dt.\end{aligned} \]
Since $|t-s|\approx2^{-j}$, we have $|\phi_{st}''|\approx 2^{-j}$ on the support of $a(t,s)$, by the properties of $c$ and $d$. Then using the mean value theorem, we get
\begin{equation}\label{est1}|\varphi'_t(t,s,s')|=|\phi''_{st}(t,s'')|\approx 2^{-j},\end{equation}
where $s''$ is between $s$ and $s'$. Here we use the fact that $t-s\approx t-s'\approx 2^{-j}$ on the support of $a(t,s)$.  Moreover, direct calculation shows
\begin{equation}\label{est2}|\partial_t^k\tilde a|\lesssim 2^{(k+2)j},\ |\partial_t^k\phi_{st}''|\lesssim 1,\ k=0,1,2.\end{equation}
One the one hand, by using \eqref{est1} and \eqref{est2} and the support property of $a(t,s)$, we have  \[|K(s,s')| \le \int {|a(t,s)||a(t,s')|dt} \lesssim 2^{j}.\] One the other hand, we can do integration by parts twice to get
\begin{align*}
  |K(s,s')| &\le {(\lambda |s - s'|)^{ - 2}}\int {\left| {\frac{\partial }
{{\partial t}}\left( {\frac{1}
{{\varphi'_t}}\frac{\partial }
{{\partial t}}\left( {\frac{{\tilde a}}
{{\varphi'_t}}} \right)} \right)} \right|dt}  \\
   &\lesssim (\lambda |s - s'|)^{ - 2}2^{5j}.
\end{align*}
Then \[\int |K(s,s')|ds\lesssim \int_0^{2^{-j}} \min\{2^j,  (\lambda s)^{ - 2}2^{5j}\}ds\lesssim \lambda^{-1}2^{3j}\]
 and Young's inequality gives $\|\tilde T_\lambda^j\|_{L^2\to L^2}\lesssim \lambda^{-\frac12}2^{\frac32j}$ for $2^j\le \lambda^\frac13$. Therefore, the uniform bound $\|T_\lambda\|_{L^2\to L^2}\ls 1$ follows from summing a geometric series. So the proof is complete.
\begin{remark}
It is worth mentioning that one may slightly modify the argument above to give a simpler different proof of the main theorems by Burq-G\'erard-Tzvetkov \cite{burq} about the eigenfunction restriction estimates for curves on Riemannian surfaces.
\end{remark}
\section{Further Discussions}
Let $\phi(t,s)=d_g(\gamma(t),\gamma(s))\sgn(t-s)$ as above. We say that the curve $\gamma$ satisfies the finite type condition if  the mixed derivative $\phi_{ts}''(t,s)$ does not vanish to infinite order. We recall the following results of Pan \cite{pan91}.
\begin{proposition}Suppose that $T_\lambda$ is defined in \eqref{Tlambda}, and $\phi_{st}''$ does not vanish to infinite order on ${\rm supp}\ a$. Then the operators $T_\lambda$ are uniformly bounded on $L^p$ to itself, for $1<p<\infty$.
\end{proposition}
\begin{corollary}If the phase $\phi(t,s)$ is real-analytic, then $T_\lambda$ are uniformly bounded on $L^p$, for $1<p<\infty$.
\end{corollary}
Following the argument before, we immediately have
\begin{corollary}\label{panco}
Let $(M,g)$ be a compact smooth 3-dimensional Riemannian manifold. Let $\gamma\subset M$ be a fixed unit-length curve satisfying the finite type condition. There exists a constant $C>0$ such that for any
$e_\lambda$, we have
\begin{equation}\label{key2}\|e_\lambda\|_{L^2(\gamma)}\le C(1+\lambda)^\frac12\|e_\lambda\|_{L^2(M)}.\end{equation}
\end{corollary}
\begin{corollary} If $(M,g)$ and $\gamma$ are both real-analytic, then \eqref{key2} holds.
\end{corollary}
 This bound \eqref{key2} is sharp and saturated by zonal functions on the sphere $S^3$, see \cite[Section 6.2]{burq}. One may also have \eqref{key2} if the finite type condition is replaced by other different conditions in \cite{seeger94}, \cite{cp99}. The following problems are still open.
 \begin{problem}\label{pb1}{\rm Is \eqref{key2} true for \emph{any} smooth curve $\gamma$ on a general 3-dimensional manifold? }
 \end{problem}
  \begin{problem}\label{pb3}{\rm Is \eqref{key0} true for some ``curved'' submanifolds of codimension 2 when $n\ge4$?
  }
   \end{problem}
 \begin{problem}\label{pb2}{\rm
   If the phase $\phi(t,s)=d_g(\gamma(t),\gamma(s))\sgn(t-s)$ is the distance function restricted to a general smooth curve $\gamma$ on a general 3-dimensional manifold, then is the operator $T_\lambda$ (defined in \eqref{Tlambda})  uniformly bounded on $L^2(\mathbb{R})$?}
 \end{problem}

If the phase function can be arbitrary, the answer to Problem \ref{pb2} is No. See the work by Nagel-Wainger \cite[Theorem 4.1]{nw} for an explicit (translation invariant) counterexample $\phi(t,s)=\psi(t-s)$. Indeed, Nagel-Wainger constructed a ``flat'' function $\psi(t)$ such that
\[m(x, y)={\rm p . v .} \int_{-1}^{1} e^{i(x t+y \psi(t))} \frac{d t}{t}\]
is unbounded on $\mathbb{R}^2$. Then the Hilbert transform along the curve $(t,\psi(t))$ is unbounded on $L^2(\mathbb{R}^2)$. So by \eqref{ht} we can see that $T_\lambda$ is not uniformly bounded on $L^2(\mathbb{R})$. Furthermore, the following more explicit problem is still open.
\begin{problem}\label{pb4}{\rm Let $\lambda\in\mathbb{R}$ and $\psi \in C^\infty(\mathbb{R})$. Is
\[m(\lambda)={\rm p . v .} \int_{-1}^{1} e^{i\lambda\psi(t)} \frac{d t}{t}\]
bounded on $\mathbb{R}$?}
  \end{problem}
 Pan \cite[Lemma 2.3]{pan1993} proved that it is bounded if $\psi(t)$ does not vanish to infinite order at $t=0$. However, if we replace the interval $[-1,1]$ in the integral by $\mathbb{R}$, the answer is No. A simple counterexample can be a smooth function vanishing in $(-\infty, -1]$ and    equal  to 1 in $[1,\infty)$. Stein-Wainger \cite{sw1970} showed that if $\psi(t)$ is a polynomial of degree $d$, then
 \[\Big|{\rm p . v .} \int_{\mathbb{R}} e^{i\lambda\psi(t)} \frac{d t}{t}\Big|\le C_d\]
 where the constant $C_d$   is independent of $\lambda$ and the coefficients of the polynomial. The best constant $C_d\approx \log d$, see \cite{pa2008} and references therein.

\bibliography{newref}

\bibliographystyle{plain}

\end{document}